\documentclass{amsart}
\usepackage{graphicx}
\usepackage{amscd}
\usepackage{amsmath}
\usepackage{amsthm}
\usepackage{amsfonts}
\usepackage{amssymb}
\usepackage{mathrsfs}
\usepackage{enumerate}
\usepackage[pagebackref,hypertexnames=false, colorlinks, citecolor=blue, linkcolor=red, pdfstartview=FitB]{hyperref}
\usepackage{amsrefs}

\usepackage[latin1]{inputenc}

\newcommand{\D}{\operatorname{\mathbb{D}}}

\newcommand{\C}{\operatorname{\mathbb{C}}}

\newcommand{\B}{\operatorname{\mathcal{B}}}
\newcommand{\A}{\operatorname{\mathcal{A}}}
\newcommand{\hil}{\operatorname{\mathcal{H}}}
\newcommand{\kil}{\operatorname{\mathcal{K}}}

\newcommand{\ol}{\overline }
\DeclareMathOperator{\id}{Id}
\DeclareMathOperator{\ran}{ran}
\newtheorem{lemma}{Lemma}[section]
\newtheorem{theorem}[lemma]{Theorem}
\newtheorem{proposition}[lemma]{Proposition}
\newtheorem{corollary}[lemma]{Corollary}
\theoremstyle{definition}

\begin{document}
\author{Rapha\"el Clou\^atre}
\address{Department of Mathematics, Indiana University, 831 East 3rd Street,
Bloomington, IN 47405} \email{rclouatr@indiana.edu}
\title{Quasisimilarity of invariant subspaces for $C_0$ operators with multiplicity two }
\begin{abstract}
For an operator $T$ of class $C_0$ with multiplicity two, we show
that the quasisimilarity class of an invariant subspace $M$ is
determined by the quasisimilarity classes of the restriction $T|M$
and of the compression $T_{M^\perp}$. We also provide a canonical form for the subspace $M$.
\end{abstract}
\maketitle

\section{Introduction}
Let $T:\hil\to \hil$ and $T:\hil'\to \hil'$ be bounded linear
operators on Hilbert spaces. If $M$ and $M'$ are invariant subspaces
for $T$ and $T'$ respectively (that is $M\subset \hil$ and
$M'\subset \hil'$ are closed subspaces such that $TM\subset M$ and
$T'M'\subset M'$), we say that $M'$ is a \textit{quasiaffine
transform} of $M$ if there exists a bounded injective operator with
dense range $X:\hil\to \hil'$ such that $XT=T' X$ and $\ol{XM}=M'$.
We write $M\prec M'$ when $M'$ is a quasiaffine transform of $M$. In
that case, we also say that $M'$ lies in the \textit{quasiaffine
orbit} of $M$. When $M\prec  M'$ and $M'\prec M$, we say that $M$
and $M'$ are \textit{quasisimilar} and write $M\sim M'$.
Quasisimilarity is clearly an equivalence relation on the class of
pairs of the form $(T,M)$, where $M$ is an invariant subspace for
the bounded linear operator $T$. In \cite{B}, Bercovici raised the
basic problem underlying our present investigation: describe the
quasisimilarity class of a given invariant subspace for an operator
of class $C_0$ (see definition in Section 2). Theorem
\ref{invsubmult1} below (see Section 2) is classical and offers a
complete and very simple answer to the problem in the case where the
operator has multiplicity one (that is the operator has a cyclic
vector). Hence, we are interested in operators of class $C_0$ with
multiplicity higher than one.

In their pioneering work (see \cite{BT}), Bercovici and Tannenbaum
considered the case where $T$ is a so-called uniform Jordan operator
(namely $T=S(\theta)\oplus S(\theta)\oplus \ldots$) with finite
multiplicity. In that case, they established that the
quasisimilarity class of $M$ is determined by the quasisimilarity
class of the restriction $T|M$ (see Section 2). Moreover, the
authors  observed that for $T=S(z^2)\oplus S(z)$, this
classification breaks down, so the corresponding result may fail if
$T$ is not uniform. Later on, it was proved in \cite{B} that this
classification of invariant subspaces of a uniform Jordan operator
holds if and only if $T|M$ has property (P). In general, the
quasisimilarity class of an invariant subspace for a uniform Jordan
operator is determined by the quasisimilarity classes of the
restriction $T|M$ and of the compression $T_{M^\perp}$ (see
\cite{BS1}).

In the context of a non-uniform Jordan operator, much less is known.
Related results for general operators of class $C_0$ can be found in \cite{B},
where it is proved that the quasisimilarity class of an invariant
subspace is determined by that of $T|M$ if and only if $T$ has property (Q).
In case where $T|M$ has multiplicity one, then the weakly
quasiaffine orbit of an invariant subspace $M$ is determined by the
quasisimilarity classes of $T|M$ and $T_{M^\perp}$ (see \cite{BS2}).

More recently, nilpotent operators of finite multiplicity have been
considered by Li and Müller in \cite{LM2}. They proved that the
quasisimilarity class of $M$ is determined by the quasisimilarity
classes of $T|M$ and $T_{M^\perp}$ when either of those operators
has multiplicity one. In addition, the authors considered a
combinatorial object (a sequence of partitions) known as a
\textit{Littlewood-Richardson sequence} which encodes the
relationships that must hold between the Jordan models of $T$, $T|M$
and $T_{M^\perp}$ (see also \cite{LM1}, \cite{BSL1} and
\cite{BSL2}). Using these objects, they prove that for multiplicity
at least three, the quasisimilarity classes of $T|M$ and
$T_{M^\perp}$ are not enough to determine the quasisimilarity class
of $M$ (so that our main theorem is sharp as far as multiplicities
are concerned). In fact, that information does not even suffice to
determine the larger equivalence class of invariant subspaces having
a fixed Littlewood-Richardson sequence. However, an easy argument
shows that in the case of multiplicity two, the knowledge of the
quasisimilarity classes of $T$, $T|M$ and $T_{M^\perp}$ is enough to
determine the Littlewood-Richardson sequence corresponding to $M$:
the so-called Littlewood-Richardson rule can be satisfied in only
one way. Hence, this case seems to involve some kind of uniqueness
which is not present for higher multiplicities. Our main result
confirms and strenghtens this observation, in fact we show that for
arbitary operators of class $C_0$ with multiplicity two, the
quasisimilarity class of an invariant subspace $M$ is determined by
that of $T|M$ and $T_{M^\perp}$. We also identify a specific
invariant subspace which can serve as a canonical space.

\section{Background and preliminaries}
We give here some background concerning operators of class $C_0$.
Let $H^\infty$ be the algebra of bounded holomorphic functions on
the open unit disc $\D$. Let $\hil$ be a Hilbert space and $T$ a
bounded linear operator on $\hil$, which we indicate by
$T\in\B(\hil)$. The operator $T$ is said to be of \textit{class
$C_0$} if there exists an algebra homomorphism $\Phi: H^\infty \to
\B(\hil)$ with the following properties:
\begin{enumerate}[(i)]
    \item $\|\Phi(u)\|\leq u$ for every $u\in H^\infty$
    \item $\Phi(p)=p(T)$ for every polynomial $p$
    \item $\Phi$ is continuous when $H^\infty$ and $\B(\hil)$ are given their respective weak-star topologies
    \item $\Phi$ has non-trivial kernel.
\end{enumerate}
We use the notation $\Phi(u)=u(T)$, which is the Sz.-Nagy--Foias
$H^\infty$ functional calculus. It is known that $\ker \Phi=m_T
H^\infty$ for some inner function $m_T$ called the \textit{minimal
function} of $T$. The minimal function is uniquely determined up to
a scalar factor of absolute value one. Given a vector $x\in \hil$,
we define its minimal function, which we denote by $m_x$, to be the
minimal function of the restriction of the operator $T$ to the
invariant subspace $\bigvee_{n=0}^\infty T^n x$. A vector $x\in
\hil$ is said to be \textit{maximal} for $T$ if $m_x$ coincides with
$m_T$ up to a scalar factor of absolute value one. A set $E\subset
\hil$ is said to be \textit{cyclic} for $T$ if
$\hil=\bigvee_{n=0}^\infty T^n E$.

\begin{theorem}[\cite{bercOTA} Theorem 2.3.6, Theorem 2.3.7]\label{maxvector}
Let $T\in \B(\hil)$ be an operator of class $C_0$. Then, the set of
maximal vectors for $T$ is a dense $G_\delta$ in $\hil$. Moreover,
given any Banach space $\kil$ and any bounded linear operator
$A:\kil\to \hil$ with the property that $A\kil$ is a cyclic set for
$T$, the set
$$\{k\in \kil: Ak \text{ is maximal for } T\}$$
 is a dense $G_\delta$ in $\kil$.
\end{theorem}

We denote by $H^2$ the Hilbert space of functions
$$f(z)=\sum_{n=0}^\infty a_n z^n$$
holomorphic on the open unit disc, equipped with the norm
$$
\|f\|^2=\sum_{n=0}^\infty |a_n|^2.
$$
Our first lemma is well known.
\begin{lemma}\label{outer}
Given $f_1,f_2,\ldots,f_n\in H^2$, there exists an outer function
$v\in H^\infty$ such that $f_1 v,f_2v,\ldots,f_n v$ all belong to
$H^\infty$ as well.
\end{lemma}
\begin{proof}
It suffices to define the absolute value of $v$ on the unit circle,
which we take to be $$1/(1+|f_1|+\ldots+|f_n|).$$
\end{proof}

Recall that given functions $u,v\in H^\infty$, we say that $u$
\textit{divides} $v$ and write $u|v$ if there exists a function
$w\in H^\infty$ such that $v=wu$. We write $u \wedge v$ for the
greatest common inner divisor of $u$ and $v$, and $u\vee v$ for
their least common inner multiple. Both of these quantities are
determined up to a scalar factor of absolute value one. If $u$ and
$v$ are inner functions such that $u|v$ and $v|u$ (or, equivalently,
$u$ and $v$ only differ by a scalar factor of absolute value one),
we write $u\equiv v$. An inner function $u\in H^\infty$ is said to
divide $f\in H^2$ if $f\in uH^2$.  We naturally denote by $f\wedge
u$ the greatest common inner divisor of $f$ and $u$. A very useful
consequence of Theorem \ref{maxvector} is the following.

\begin{theorem}[\cite{bercOTA} Theorem 3.1.14]\label{ell1}
Let $\{f_j\}_{j=0}^\infty \subset H^2$ be a bounded sequence of functions, and let $\theta\in H^\infty$ be an inner function. Then, the set of sequences $\{a_j\}_{j=0}^\infty$ in $\ell^1$ satisfying
$$
\left(\sum_{j=0}^\infty a_j f_j \right)\wedge \theta\equiv\bigwedge_{j=0}^\infty f_j \wedge \theta
$$
is a dense $G_\delta$.
\end{theorem}

For any inner function $\theta\in H^\infty$, the space $H(\theta)=H^2\ominus \theta H^2$ is closed and invariant for $S^*$, the adjoint of the shift operator $S$ on $H^2$. The operator $S(\theta)$ defined by $S(\theta)^*=S^*|(H^2\ominus \theta H^2)$ is called a \textit{Jordan block}; it is of class $C_0$ with minimal function $\theta$. We give some useful properties of these operators; they will be used repeatedly throughout and often without explicit mention.

\begin{proposition}[\cite{bercOTA} Proposition 3.1.10, Corollary 3.1.12]\label{jordanblockprops}
Let $\theta\in H^\infty$ be an inner function.
\begin{enumerate}
    \item[\rm{(i)}] The operator $S(\theta)$ has multiplicity one. In fact, $h\in H(\theta)$ is cyclic if and only if $h\wedge \theta\equiv 1$.
    \item[\rm{(ii)}] If $\phi\in H^\infty$ is an inner divisor of $\theta$, then $\phi H^2\ominus \theta H^2$ is an invariant subspace for      $S(\theta)$. In fact,
    $$
    \phi H^2\ominus \theta H^2=\ran \phi(S(\theta))=\ker (\theta/\phi)(S(\theta)).
    $$
    Conversely, any invariant subspace for $S(\theta)$ is of this form.
    \item[\rm{(iii)}] Let $u\in H^\infty$ be any function and let $X=u(S(\theta))$. Then, $\ker X=\ker X^*=\{0\}$ if and only if $u\wedge \theta\equiv1$.
\end{enumerate}
\end{proposition}

The following result follows from the commutant lifting theorem (see \cite{Sar}).
\begin{theorem}[\cite{bercOTA} Theorem 3.1.16]\label{commutantJordan}
Let $\theta_0$ and $\theta_1$ be two inner functions. Assume that $X:H(\theta_0)\to H(\theta_1)$ satisfies $XS(\theta_0)=S(\theta_1)X$. Then, there exists a function $u\in H^\infty$ such that $\theta_1|u \theta_0$, $\|u\|=\|X\|$ and
$$
X=P_{H(\theta_1)}u(S)|H(\theta_0).
$$
Conversely, given any function $u\in H^\infty$ satisfying $\theta_1|u \theta_0$, the operator $P_{H(\theta_1)}u(S)|H(\theta_0)$ intertwines $S(\theta_1)$ and $S(\theta_0)$.
\end{theorem}

A more general family of operators consists of the so-called
\textit{Jordan operators}. We will define them here in the case
where the Hilbert space on which they act is separable. These
operators are of the form $\bigoplus_{j=0}^\infty S(\theta_j)$ where
$\{\theta_j\}_{j=0}^\infty$ is a sequence of inner functions
satisfying $\theta_{j+1}|\theta_j$ for $j\geq 0$. The Jordan
operators are of fundamental importance in the study of operators of
class $C_0$ as the following theorem illustrates. Recall first that
a bounded injective linear operator with dense range is called a
\textit{quasiaffinity}. Two operators $T\in \B(\hil)$ and $T'\in
\B(\hil')$ are said to be \textit{quasisimilar} if there exist
quasiaffinities $X:\hil\to \hil'$ and $Y:\hil' \to \hil$ such that
$XT=T' X$ and $T Y=YT'$. We use the notation $T\sim T'$ to indicate
that $T$ and $T'$ are quasisimilar. Recall also that the
\textit{multiplicity} of an operator is the smallest cardinality of
a cyclic set.

\begin{theorem}[\cite{bercOTA} Theorem 3.5.1, Corollary 3.5.25]\label{existencejordan}
For any operator $T$ of class $C_0$ acting on a separable Hilbert
space there exists a unique Jordan operator $J=\bigoplus_{j=0}^\infty S(\theta_j)$ such that $T$ and
$J$ are quasisimilar. Moreover, for each $j\geq 0$ we have
$$
\theta_j\equiv\bigwedge\{\phi\in H^\infty: T|\ol{\ran \phi(T)}
\text{ has multiplicity at most } j\}.
$$
\end{theorem}

The operator $J$ in the previous theorem is called the \textit{Jordan model} of $T$.

\begin{corollary}[\cite{bercOTA} Theorem 3.4.12]\label{multiplicityn}
Let $T$ be an operator of class $C_0$ with Jordan model
$\bigoplus_{j=0}^\infty S(\theta_j)$. Then, $T$ has multiplicity at
most $n$ if and only if $\theta_n\equiv1$.
\end{corollary}

\begin{theorem}[\cite{bercOTA} Corollary 3.1.7, Proposition 3.5.30]\label{jordanadjoint}
Let $T$ be an operator of class $C_0$ with Jordan model
$\bigoplus_{j=0}^\infty S(\theta_j)$. Then, $T^*$ is also of class
$C_0$ and its Jordan model is $\bigoplus_{j=0}^\infty
S(\theta_j^{\sim})$, where for any function $u\in H^\infty$ we
define $u^{\sim}(z)=\ol{u(\ol{z})}$ for $z\in \D$.
\end{theorem}

The next result will be crucial for us; it is usually referred to as
the splitting principle.

\begin{theorem}[\cite{B} Proposition 1.17]\label{splitting}
Let $T\in \B(\hil)$ be an operator of class $C_0$ with Jordan model
$S(\theta_0)\oplus S(\theta_1)$. Let $K\subset \hil$ be an invariant subspace such that $T|K\sim S(\theta_0)$. Let $k\in K$ be
a cyclic vector for $(T|K)^*$ and set $K'=\bigvee_{n=0}^\infty
T^{*n}k$, $L=\hil\ominus K'$. Then, $\hil=K\vee L$, $K\cap L=\{0\}$
and $T|L\sim S(\theta_1)$.
\end{theorem}

To an operator of class $C_0$ with finite multiplicity we can
associate an inner function called its \textit{determinant}: if
$\bigoplus_{j=0}^\infty S(\theta_j)$ is the Jordan model of the
operator $T$ which has multiplicity $n$, then $\det T=\theta_0
\theta_1\cdots \theta_{n-1}$. The following result is helpful when
calculating the functions appearing in the Jordan model of an
operator. Given an invariant subspace $M$ for an operator
$T$, we denote by $T_{M^\perp}$ the compression
$P_{M^\perp}T|M^\perp$.

\begin{theorem}[\cite{bercOTA} Theorem 7.1.4]\label{det}
Let $T$ be an operator of class $C_0$ with finite multiplicity, and let $M$ be an invariant subspace for $T$. Then, $\det T=\det (T|M) \det(T_{M^\perp})$.
\end{theorem}

A consequence is the following.

\begin{theorem}[\cite{bercOTA} Remark 7.1.15, Proposition 7.1.21]\label{latticeisom}
Let $T\in \B(\hil)$ and $T'\in\B(\hil')$ be two operators of class
$C_0$ with finite multiplicities such that $\det T=\det T'$. Let
$X:\hil\to \hil'$ be a bounded linear operator such that $XT=T'X$.
Then, $X$ is one-to-one if and only if it has dense range.
\end{theorem}

We now collect some facts about invariant subspaces for operators of class $C_0$.

\begin{theorem}[\cite{bercOTA} Theorem 3.2.13]\label{invsubmult1}
Let $T$ be an operator of class $C_0$ with multiplicity one. Then, for every inner divisor $\theta$ of $m_T$, there exists a unique invariant subspace $M$ such that $T|M\sim S(\theta)$, namely $M=\ker \theta(T)=\ol{\ran (m_T/\theta)(T)}$.
\end{theorem}

Given an operator $T\in \B(\hil)$, we denote its commutant by  $\{T\}'=\{X\in \B(\hil):XT=TX\}$.
Recall that a closed subspace $M$ is said to be \textit{hyperinvariant} for $T$ if it is invariant for every operator $X\in \{T\}'$.

\begin{theorem}[\cite{bercOTA} Proposition 4.2.1]\label{hyperinvJordan}
Let $\bigoplus_{j=0}^\infty S(\theta_j)$ be a Jordan operator. Then,
a subspace $M$ is hyperinvariant if and only if it is of the form
$$
K=\bigoplus_{j=0}^\infty \left(\psi_j H^2\ominus \theta_j H^2\right)
$$
where $\psi_j|\theta_j, \psi_{j+1}|\psi_j$ and $(\theta_{j+1}/\psi_{j+1})|(\theta_j/\psi_j)$ for every $j\geq 0$.
\end{theorem}

Let us mention elementary facts about matrices of operators. Here and throughout, we identify a function $u\in H^\infty$ with the multiplication operator $f\mapsto uf$ it defines on $H^2$.

\begin{lemma}\label{invert}
Let  $\theta_0,\theta_1\in H^\infty$ be two inner functions such that $\theta_1|\theta_0$. Let $\A\subset M_2(H^\infty)$ be the subalgebra of matrices $A=(a_{ij})_{i,j=0,1}$ such that $\theta_0/\theta_1$ divides $a_{01}$. Let $\hil=H(\theta_0)\oplus H(\theta_1)$ and $T=S(\theta_0)\oplus S(\theta_1)$.
Then, the map
$$
\Psi: \A\to \{T\}'
$$
$$
\Psi(A)= P_{\hil}A|\hil
$$
is a surjective algebra homomorphism, with
$$
\ker \Psi=
\left\{
\left(\begin{array}{cc}
\theta_0 b_{00} & \theta_0 b_{01}\\
\theta_1 b_{10} & \theta_1 b_{11}\end{array}\right): b_{ij}\in H^\infty, i,j=0,1
\right\}.
$$
For every $A\in \A$, there exists $A'\in \A$ such that the operator $\Psi(A')$ satisfies $\Psi(A)\Psi(A')=\Psi(A')\Psi(A)=u(T)$, where $u=\det A$. When $\det A \wedge \theta_0 \equiv1$ we have that $\Psi(A)$ and $\Psi(A')$ are quasiaffinities.
\end{lemma}

\begin{proof}
Theorem \ref{commutantJordan} ensures that $\Psi$ maps $\A$ onto
$\{T\}'$. Moreover, it is clear that $\Psi$ is linear and that it
has the announced kernel. It remains to show that it is
multiplicative. First note that if $b\in H^\infty$ is divisible by
$\theta_0/\theta_1$, then
$$
P_{H(\theta_0)}bP_{H(\theta_1)} af=P_{H(\theta_0)}b(af-
P_{\theta_1H^2}af)=P_{H(\theta_0)}baf
$$
for every $a\in H^\infty, f\in H^2$. Moreover, since $H(\theta_1)\subset H(\theta_0)$, we have
\begin{align*}
P_{H(\theta_1)}bP_{H(\theta_0)}a|H(\theta_0)&=P_{H(\theta_1)}P_{H(\theta_0)}
b P_{H(\theta_0)}
a|H(\theta_0)\\
&=P_{H(\theta_1)}P_{H(\theta_0)}ba|H(\theta_0)\\
&=P_{H(\theta_1)}ba|H(\theta_0)
\end{align*}
for every $b,a\in H^\infty$. These considerations show that
$P_{\hil}(BA)|\hil=(P_{\hil}B|\hil)(P_{\hil}A|\hil)$ for $A\in
M_2(H^\infty)$ and $B\in \A$, so that $\Psi$ is an algebra
homomorphism.

Given $A\in \A$, let $A'$ be the (pointwise) algebraic adjoint of
$A$, so that $AA'=A'A=u \id_{\C^2}$ with $u=\det A$. Note that
$a'_{01}=-a_{01}$ so that $A'\in \A$. Set $X=\Psi(A)$ and
$X'=\Psi(A')$. We have $X'X=\Psi(A'A)=u(T)=\Psi(AA')=XX'$, that is
$$
XX'=XX'=  \left(
\begin{array}{cc}
u(S(\theta_0))& 0\\
0 & u(S(\theta_1))
\end{array}
\right).
$$
When $u \wedge \theta_0\equiv 1$, the operator $u(T)$ is a
quasiaffinity by Proposition \ref{jordanblockprops}. In particular,
this shows that $X$ and $X'$ are quasiaffinities.
\end{proof}

Let us recall a relation between operators which is weaker than that
of quasisimilarity. Given $T\in \B(\hil)$ and $T'\in \B(\hil')$, we
say that $T'$ is a \textit{quasiaffine transform} of $T$  and we
write $T\prec T'$, if there exists a quasiaffinity $X:\hil\to \hil'$
such that $XT=T'X$.

\begin{theorem}[\cite{bercOTA} Proposition 3.5.32]\label{injection}
Let $T$ and $T'$ be two operators of class $C_0$. Then, $T\prec T'$
is equivalent to $T\sim T'$.
\end{theorem}

The next proposition can rephrased as saying that any multiplicity
two operator of class $C_0$ has the so-called property $(*)$ (see \cite{B}).

\begin{proposition}[\cite{bercOTA} Lemma 4.1.11, Proposition 4.1.13]\label{prop*}
Let $T$ be an operator of class $C_0$ with multiplicity two. Let
$X\in \{T\}'$ be a quasiaffinity. Then, there exists another
quasiaffinity $Y\in \{T\}'$ and a function $u\in H^\infty$ such that
$XY=YX=u(T)$.
\end{proposition}

We obtain a useful consequence.

\begin{lemma}\label{similar}
Let $T$ be an operator of class $C_0$ with multiplicity
two. Let $M, M'$ be invariant subspaces for $T$. Assume
that there exist quasiaffinities $X,X'\in\{T\}'$ such that
$\ol{XM}=\ol{X'M'}$. Then, $M\sim M'$.
\end{lemma}
\begin{proof}
By Proposition \ref{prop*}, there exist quasiaffinities $Y,Y'\in
\{T\}'$ and functions $u,u'\in H^\infty$ such that $XY=YX=u(T)$ and $X'Y'=Y'X'=u'(T)$. Hence, $$\ol{Y X'
M'}=\ol{Y X M}=\ol{u(T)M}\subset M$$ and
$$
\ol{Y' X M}=\ol{Y' X' M'}=\ol{u'(T)M'}\subset M'.
$$
On the other hand, by Theorem \ref{injection} we have that
$$
T|M'\sim T|\ol{X' M'}= T|\ol{XM}\sim T|\ol{Y' X M}
$$
and
$$
T|M\sim T|\ol{X M}= T|\ol{X'M'}\sim T|\ol{Y X' M'}
$$
so
that Theorem \ref{latticeisom} implies that $\ol{Y X' M'}=M$
and $\ol{Y' X M}=M'$, and we are done.
\end{proof}

Let us close this section by proving an elementary fact which
motivates our main result.

\begin{proposition}\label{jordansim}
Let $T\in \B(\hil)$ and $T'\in \B(\hil')$ be operators of class $C_0$ with finite multiplicities. Assume that $X:\hil \to \hil'$ is a quasiaffinity such that $XT=T'X$. Let $M\in \hil$ be an invariant subspace for $T$. Then, $T'|\ol{XM}\sim T|M$ and $T'_{(XM)^\perp}\sim T_{M^\perp}$.
\end{proposition}
\begin{proof}
We clearly have $T\prec T'$, so by Theorem \ref{injection} we find
$T\sim T'$. Hence $\det T=\det T'$. If we let $E=\ol{XM}$, it
follows from Theorem \ref{injection} again that $T'|E\sim T|M$, and
in particular $\det(T|M)=\det(T'|E)$. Using Theorem \ref{det}, we
find
$$
\det(T_{M^\perp})=\frac{\det T}{\det (T|M)}=\frac{\det T'}{\det (T'|E)}=\det(T'_{E^\perp}).
$$
Moreover, we have $X^* E^\perp\subset M^\perp$ and $X^* T'^*=T^*
X^*$. Since $X^*$ is injective, we may apply Theorem
\ref{latticeisom} to find $\ol{X^* E^\perp}=M^\perp$. This
establishes $T'^*|E^\perp \prec T^*|M^\perp$. By Theorem
\ref{injection}, we have $T'^*|E^\perp \sim T^*|M^\perp$ and
$T'_{E^\perp} \sim T_{M^\perp}$.
\end{proof}

Proposition \ref{jordansim} shows in particular that if $T$ is an operator
of class $C_0$ with multiplicity two and $M,M'$ are invariant subspaces for $T$, then $M\sim M'$ implies $T|M\sim
T|M'$ and $T_{M^\perp}\sim T_{M'^\perp}$. Our main theorem says that
the converse holds.

\section{Jordan model and hyperinvariance}
Let us first make a convention. Let $\theta_0,\theta_1\in H^\infty$ be inner functions such that $\theta_1$ divides $\theta_0$. In the space $H(\theta_0)\oplus H(\theta_1)$, we identify the subspace $H(\theta_0)\oplus \{0\}$ with $H(\theta_0)$ and the subspace $\{0\}\oplus H(\theta_1)$ with $H(\theta_1)$. Given $M\subset H(\theta_0)\oplus H(\theta_1)$ an invariant subspace for the Jordan operator $S(\theta_0)\oplus S(\theta_1)$, it is easy to verify that $\ol{P_{H(\theta_j)}M}$ is invariant for $S(\theta_j)$ for each $j=0,1$, so by Proposition \ref{jordanblockprops} we can find $\phi_j \in H^\infty$ an inner divisor of $\theta_j$ such that $\ol{P_{H(\theta_j)} M}=\phi_j H^2\ominus \theta_j H^2$. We will use this observation implicitly throughout the remainder of the paper. The following result allows us to focus on a very special kind of subspace $M$; it is a spiritual cousin of Theorem 3.2 in \cite{BS2}.

\begin{theorem}\label{hyper}
Let $T=S(\theta_0)\oplus S(\theta_1)$ be a Jordan operator and $M$ be an invariant subspace for $T$. Then, there exists a quasiaffinity $X\in \{T\}'$ such that
$$
\ol{P_{H(\theta_0)}XM}\oplus \ol{P_{H(\theta_1)}XM}
$$
is a hyperinvariant subspace for $T$.
\end{theorem}

\begin{proof}
Let $\xi \in M$ be a maximal vector for $T|M$ and write
$\xi=\xi_0\oplus \xi_1\in H(\theta_0)\oplus H(\theta_1)$. By Theorem
\ref{ell1}, we find can non-zero $a_0,a_1\in \C$ such that
$$
\left(a_0\xi_0+a_1\frac{\theta_0}{\theta_1}\xi_1\right)\wedge\theta_0\equiv \xi_0
\wedge \frac{\theta_0}{\theta_1}\xi_1 \wedge \theta_0
$$
along with non-zero $b_0,b_1\in \C$ such that
$$
(b_0 \xi_0 +b_1 \xi_1)\wedge \theta_1 \equiv \xi_0 \wedge \xi_1 \wedge \theta_1
$$
and
$$
\left(a_0 b_1-b_0 a_1\frac{\theta_0}{\theta_1} \right)\wedge \theta_0\equiv a_0\wedge
a_1 \frac{\theta_0}{\theta_1}\wedge \theta_0\equiv 1\wedge
\frac{\theta_0}{\theta_1}\wedge \theta_0\equiv 1.
$$
Set
$$
A=\left(
\begin{array}{cc}
a_0 & a_1\theta_0/\theta_1 \\
b_0 & b_1
\end{array}
\right)
$$
and
$$
X=P_{H(\theta_0)\oplus H(\theta_1)}A|(H(\theta_0)\oplus
H(\theta_1)).
$$
Now, $\det A=a_0b_1-b_0 a_1\theta_0/\theta_1$, so that $X$ is a
quasiaffinity commuting with $T$ by Lemma \ref{invert}. Notice that
$$
m_{T|M}\equiv m_{\xi}\equiv \frac{\theta_0}{\xi_0\wedge \theta_0}\vee
\frac{\theta_1}{\xi_1\wedge \theta_1}\equiv\frac{\theta_0}{\xi_0\wedge
(\theta_0/\theta_1)\xi_1\wedge \theta_0}.
$$
Therefore, if we let
$X\xi=y_0\oplus y_1\in
H(\theta_0)\oplus H(\theta_1)$
we find
$$
 y_0\wedge
\theta_0\equiv\left(a_0 \xi_0+a_1 \frac{\theta_0}{\theta_1}\xi_1\right)\wedge
\theta_0\equiv \xi_0 \wedge \frac{\theta_0}{\theta_1}\xi_1 \wedge
\theta_0\equiv\frac{\theta_0}{m_{T|M}},
$$
and
$$
 y_1\wedge \theta_1\equiv(b_0\xi_0+b_1\xi_1)\wedge \theta_1\equiv \xi_0\wedge \xi_1 \wedge \theta_1
$$
so that $y_1\wedge \theta_1$ divides $\theta_0/m_{T|M}$.

Consider now the hyperinvariant subspace generated by $M$,
$$
E=\bigvee \{YM:Y\in \{T\}' \}.
$$
By Theorem \ref{hyperinvJordan}, $E$ can be written as
$$
E=(\psi_0 H^2\ominus \theta_0 H^2) \oplus (\psi_1 H^2\ominus \theta_1 H^2)
$$
for some inner functions $\psi_0,\psi_1\in H^\infty$ with the
property that $\psi_1|\psi_0$, $(\theta_1/\psi_1)|(\theta_0/\psi_0)$ and $\psi_j|\theta_j$ for $j=0,1$.
Note that for $Y\in \{T\}'$, $u\in H^\infty$ and $h\in
H(\theta_0)\oplus H(\theta_1)$, we have
$$
u(T)Yh=Yu(T)h
$$
which shows that $m_{T|M}(T) Yh=0$ for every $h\in M$ and $Y\in
\{T\}'$, whence the minimal functions of $T|E$ and $T|M$ coincide,
and
$$
\theta_0/\psi_0\equiv m_{T|E}\equiv m_{T|M}.
$$
For each $j=0,1,$ write
$$
\ol{P_{H(\theta_j)}XM}=\phi_j H^2\ominus \theta_j H^2
$$
where $\phi_j$ is an inner divisor of $\theta_j$. Since $XM\subset E$, we have
$$
\phi_j H^2\ominus \theta_j H^2=\ol{P_{H(\theta_j)}XM}\subset \ol{P_{H(\theta_j)} E}=\psi_j H^2\ominus \theta_j H^2
$$
and thus $\psi_j$ divides $\phi_j$ for $j=0,1$. Notice that $y_0\in
P_{H(\theta_0)}XM$, so $\phi_0$ divides $y_0\wedge \theta_0$. But we
established above that $y_0\wedge
\theta_0\equiv\theta_0/m_{T|M}\equiv\psi_0$, so we find
$\phi_0\equiv\psi_0$. In addition, $y_1\in P_{H(\theta_1)}XM$
implies that $\phi_1$ divides $y_1\wedge \theta_1$, which in turn
divides $\theta_0/m_{T|M}$ as was shown above. Since
$\theta_0/m_{T|M}\equiv\psi_0\equiv\phi_0$, we have that $\phi_1$
divides $\phi_0$. Finally, using the fact that $\psi_1$ divides
$\phi_1$, we find that $\theta_1/\phi_1$ divides $\theta_1/\psi_1$,
which in turn divides $\theta_0/\psi_0\equiv\theta_0/\phi_0$.
Theorem \ref{hyperinvJordan} completes the proof.
\end{proof}

The following is based on Theorem 3.4 of \cite{BS2}. Properties
(iii) and (iv) below are part of the so-called \textit{Weyl
identities} (see \cite{BL}).

\begin{proposition}\label{jordanmodel}
Let $T=S(\theta_0)\oplus S(\theta_1)$ be a Jordan operator. Let $M$ be an invariant subspace for $T$. For each $j=0,1$, write
$$
\ol{P_{H(\theta_j)} M}=\phi_j H^2\ominus \theta_j H^2
$$
where $\phi_j$ is an inner divisor of $\theta_j$. Assume that $\phi_1|\phi_0$ and $(\theta_1/\phi_1)|(\theta_0/\phi_0)$. Let $S(\alpha_0)\oplus S(\alpha_1)$ be the Jordan model of $T|M$ and $S(\beta_0)\oplus S(\beta_1)$ be the Jordan model of $T_{M^\perp}$. Then,
\begin{enumerate}
    \item[\rm{(i)}]  $\theta_0 \theta_1=\alpha_0 \alpha_1 \beta_0 \beta_1$
    \item[\rm{(ii)}] $\phi_0\equiv\theta_0/\alpha_0$ and $\phi_1\equiv\beta_1$
    \item[\rm{(iii)}] $(\theta_1/\beta_1)|\alpha_0$
    \item[\rm{(iv)}] $\beta_1|(\theta_0/\alpha_0)$ and
    $\beta_1|(\theta_1/\alpha_1)$.
\end{enumerate}
\end{proposition}

\begin{proof}
Using the decomposition $\hil=M\oplus M^\perp$ to compute the determinant, we find by Theorem \ref{det} that
$$
\theta_0\theta_1=\det T=\det (T|M)\det (T_{M^\perp})=\alpha_0 \alpha_1 \beta_0\beta_1
$$
which is (i). Assume for the moment that (ii) holds. Then, (iii) is equivalent to our assumption $(\theta_1/\phi_1)|(\theta_0/\phi_0)$. Moreover, $\phi_1|\phi_0$ is equivalent to the first part of (iv). Note now that if $h=h_0\oplus h_1\in M$, then by choice of $\phi_1$ we have $(\theta_1/\phi_1)(T)h\in H(\theta_0)\oplus \{0\}$, so in fact the operator
$$
(T|M)|\ol{\ran (\theta_1/\phi_1)(T|M)}
$$
has multiplicity at most $1$. By Theorem \ref{existencejordan}, we conclude that $\alpha_1|(\theta_1/\phi_1)$ so the second half of (iv) also follows from (ii). Hence, it only remains to identify $\phi_0$ and $\phi_1$, that is to show (ii).

Since $T|M$ has multiplicity two, we can find vectors $\xi=\xi_0
\oplus \xi_1\in M$ and $\eta=\eta_0\oplus \eta_1 \in M$ with
$\xi_j,\eta_j\in H(\theta_j)$ for $j=0,1$,  such that
$$
M=\bigvee_{n=0}^\infty \left\{T^n \xi, T^n \eta \right\}.
$$
It is easy to see that $\phi_j\equiv\xi_j\wedge \eta_j \wedge
\theta_j$  and
$$
\ol{P_{H(\theta_j)} M}=\bigvee_{n=0}^\infty \left\{S(\theta_j)^n
\xi_j, S(\theta_j)^n \eta_j \right\}
$$
for $j=0,1$. Now, we have
$$
S(\theta_j)|\ol{P_{H(\theta_j)} M}\sim S(\theta_j/\phi_j)
$$
so that $\theta_j/\phi_j$ is the minimal function of
$S(\theta_j)|\ol{P_{H(\theta_j)} M}$. Since for every $u\in
H^\infty$ we have
$$
u(T)=u(S(\theta_0))\oplus u(S(\theta_1)),
$$
it is clear that the minimal function of $T|M$, namely $\alpha_0$, is equal to the least common inner multiple of the minimal functions of  $S(\theta_j)|\ol{P_{H(\theta_j)} M}$ for $j=0,1$, which is $\theta_0/\phi_0$. Hence, $\alpha_0\equiv\theta_0/\phi_0$.

Consider the operators $Y=P_{H(\theta_0)}|M$ and
$Z=P_{H(\theta_1)}|M$. It is straightforward to verify that
$$
YT_{M\ominus \ker Y}=(S(\theta_0)|\ol{YM})(Y|(M\ominus \ker Y))
$$
and
$$
ZT_{M\ominus \ker Z}=(S(\theta_1)|\ol{ZM}))(Z|(M\ominus \ker Z)).
$$
Thus, by virtue of Theorem \ref{injection}, we get
$$
T_{M\ominus \ker Y}\sim S(\theta_0)|\ol{P_{H(\theta_0)} M}\sim S(\theta_0/\phi_0)
$$
and\
$$
T_{M\ominus \ker Z}\sim S(\theta_1)|\ol{P_{H(\theta_1)} M}\sim S(\theta_1/\phi_1).
$$
Notice now that
$$
\ker Y=M\cap (0\oplus H(\theta_1))=0\oplus (\rho H^2\ominus \theta_1H^2)
$$
and
$$
\ker Z=M\cap (H(\theta_0)\oplus 0)=(\sigma H^2\ominus \theta_0 H^2)\oplus 0
$$
for some inner functions $\rho,\sigma \in H^\infty$ such that $\rho$
divides $\theta_1$ and $\sigma$ divides $\theta_0$. We now calculate
the determinant using the decomposition $M=\ker Y\oplus (M\ominus
\ker Y)$, namely
$$
\alpha_0 \alpha_1=\det (T|M)=\det (T|\ker Y)\det (T_{M\ominus \ker
Y})=\frac{\theta_0 \theta_1}{ \phi_0 \rho}.
$$
Using the fact that $\alpha_0\equiv\theta_0/\phi_0$, we find
$\alpha_1\equiv\theta_1/\rho$. On the other hand, writing $M=\ker
Z\oplus (M\ominus \ker Z)$ and computing determinants, we find
$$
\alpha_0 \alpha_1=\det (T|M)=\det (T|\ker Z)\det (T_{M\ominus \ker Z})=\frac{\theta_0 \theta_1}{\sigma \phi_1}.
$$
whence
$$\sigma=\frac{\theta_0 \theta_1}{\alpha_0 \alpha_1 \phi_1}.$$
Using once again that $\alpha_0\equiv\theta_0/\phi_0$, we get
$$\sigma\equiv\rho \frac{\phi_0 }{\phi_1}.$$
Note that by assumption $\phi_1$ divides $\phi_0$, so that $\rho$
divides $\sigma$. Now, $\beta_0$ is the minimal function of
$T_{M^\perp}$, and thus is the greatest common inner divisor of the
functions $u\in H^\infty$ such that $u(T)(H(\theta_0)\oplus
H(\theta_1))\subset M$. Given such a function $u$, we have
$$
u(T)(P_{H(\theta_0)}1\oplus 0)=P_{H(\theta_0)} u\oplus 0\in M\cap (H(\theta_0)\oplus 0)=\ker Z
$$
and thus $\sigma$ divides $u$. Conversely, using that $\rho$ divides $\sigma$, we have
\begin{align*}
\sigma(T)(H(\theta_0)\oplus H(\theta_1)) &= \ran \sigma(S(\theta_0))\oplus \ran \sigma (S(\theta_1)\\
&\subset\ran \sigma(S(\theta_0))\oplus \ran \rho(S(\theta_1))\\
&=(\sigma H^2\ominus \theta_0 H^2)\oplus (\rho H^2\ominus \theta_1 H^2)\\
&=\ker Z\vee \ker Y\subset M.
\end{align*}
Hence, $\beta_0\equiv\sigma$. Finally, using the decomposition $\hil=M\oplus M^\perp$ to compute the determinant, we find
$$
\theta_0\theta_1=\det T=\det (T|M)\det (T_{M^\perp})=\alpha_0 \alpha_1 \beta_0\beta_1
$$
and thus
$$
\beta_1\equiv\frac{\theta_0 \theta_1 }{\alpha_0 \alpha_1 \beta_0}\equiv\frac{\phi_0 \rho}{\sigma}\equiv\phi_1.
$$
\end{proof}

We close this section by providing a type of converse to Proposition \ref{jordanmodel}. Let us first establish an elementary lemma.

\begin{lemma}\label{cycliccontained}
Let $T=S(\theta_0)\oplus S(\theta_1)$ be a Jordan operator. Assume that $\phi_0,\phi_1\in H^\infty$ are inner divisors of $\theta_0$ and $\theta_1$ respectively, with the additional property that either $(\theta_1/\phi_1)|(\theta_0/\phi_0)$ or $(\theta_0/\phi_0)|(\theta_1/\phi_1)$. Set
$$
\xi=(P_{H(\theta_0)}\phi_0)\oplus (P_{H(\theta_1)}\phi_1)\in H(\theta_0)\oplus H(\theta_1).
$$
Then
$$
\bigvee_{n=0}^\infty T^n \xi\subset \left\{(P_{H(\theta_0)}\phi_0 g)\oplus (P_{H(\theta_1)}\phi_1 g): g\in H^2 \right\}.
$$
\end{lemma}
\begin{proof}
Let $y=y_0\oplus y_1\in \bigvee_{n=0}^\infty T^n \xi$. We can find a
sequence of polynomials $\{r_n\}_n$ with the property that for each
$j=0,1,$ we have
$$
y_j=\lim_{n\to \infty} r_n(S(\theta_j))(P_{H(\theta_j)}\phi_j).
$$
Notice now that for each $n\geq 0$, we have
$$
r_n(S(\theta_j))(P_{H(\theta_j)}\phi_j)=P_{H(\theta_j)}r_n
\phi_j=\phi_j
P_{H(\theta_j/\phi_j)}r_n=\phi_j(S(\theta_j))(P_{H(\theta_j)}r_n)
\in \ran \phi_j( S(\theta_j)).
$$
Proposition \ref{jordanblockprops} implies that the range of
$\phi_j(S(\theta_j))$ is closed, so there exist $g_j\in H(\theta_j)$
with the property that $y_j=\phi_j(S(\theta_j))g_j=\phi_j
P_{H(\theta_j/\phi_j)}g_j$. We have for $j=0,1$ that
$$
\phi_j P_{H(\theta_j/\phi_j)}g_j=\lim_{n\to \infty }\phi_j
P_{H(\theta_j/\phi_j)}r_n
$$
and since $1/\phi_j=\ol{\phi_j}$  we find that
\begin{equation}\label{relgj}
P_{H(\theta_j/\phi_j)}g_j=\lim_{n\to \infty
}P_{H(\theta_j/\phi_j)}r_n.
\end{equation}
If $(\theta_1/\phi_1)|(\theta_0/\phi_0)$, we have $H(\theta_1/\phi_1)\subset H(\theta_0/\phi_0)$ and (\ref{relgj}) implies that $P_{H(\theta_1/\phi_1)}g_0=P_{H(\theta_1/\phi_1)}g_1$. Hence
$$
y_1=\phi_1 P_{H(\theta_1/\phi_1)}g_1=\phi_1 P_{H(\theta_1/\phi_1)}g_0=P_{H(\theta_1)}\phi_1 g_0
$$
so that
$$
y_0\oplus y_1\in \left\{(P_{H(\theta_0)}\phi_0 g)\oplus (P_{H(\theta_1)}\phi_1 g): g\in H^2 \right\}.
$$
If, on the other hand, $(\theta_0/\phi_0)|(\theta_1/\phi_1)$, we have $H(\theta_0/\phi_0)\subset H(\theta_1/\phi_1)$ and (\ref{relgj}) implies that $P_{H(\theta_0/\phi_0)}g_1=P_{H(\theta_0/\phi_0)}g_0$. Hence
$$
y_0=\phi_0 P_{H(\theta_0/\phi_0)}g_0=\phi_0 P_{H(\theta_0/\phi_0)}g_1=P_{H(\theta_0)}\phi_0 g_1
$$
so that
$$
y_0\oplus y_1\in \left\{(P_{H(\theta_0)}\phi_0 g)\oplus (P_{H(\theta_1)}\phi_1 g): g\in H^2 \right\}.
$$
\end{proof}

\begin{proposition}\label{canonicalspace}
Let $\theta_j,\alpha_j,\beta_j\in H^\infty$ be inner functions such that $\alpha_j|\theta_j$ and $\beta_j|\theta_j$ for $j=0,1$. Assume that
\begin{enumerate}
    \item[\rm{(i)}] $\theta_1|\theta_0$, $\beta_1|\beta_0$ and $\alpha_1|\alpha_0$
    \item[\rm{(ii)}] $(\theta_1/\beta_1)|\alpha_0$
    \item[\rm{(iii)}] $\beta_1|((\theta_0/\alpha_0)\wedge (\theta_1/\alpha_1))$
    \item[\rm{(iv)}]  $\theta_0 \theta_1=\alpha_0 \alpha_1 \beta_0 \beta_1$.
\end{enumerate}
Let $T=S(\theta_0)\oplus S(\theta_1)$. If we set
$$
\xi=(P_{H(\theta_0)}(\theta_0/\alpha_0))\oplus (P_{H(\theta_1)}\beta_1)
$$
and
$$
\eta=0\oplus (P_{H(\theta_1)}(\theta_1/\alpha_1)),
$$
then
$$
N=\bigvee_{n=0}^\infty \{T^n \xi, T^n \eta\}
$$
is an invariant subspace for $T$ with the property that
$$
\ol{P_{H(\theta_0)}N}=(\theta_0/\alpha_0) H^2\ominus \theta_0 H^2,
$$
$$\ol{P_{H(\theta_1)}N}=\beta_1 H^2\ominus \theta_1 H^2,
$$
$$T|N\sim S(\alpha_0)\oplus S(\alpha_1)$$
and
$$
T_{N^\perp}\sim S(\beta_0)\oplus S(\beta_1).
$$
\end{proposition}
\begin{proof}
Let $\phi_0=\theta_0/\alpha_0$, $\phi_1=\beta_1$, $\psi_0=\theta_0$
and $\psi_1=\theta_1/\alpha_1$. We have $\eta=(P_{H(\theta_0)}\psi_0
)\oplus (P_{H(\theta_1)}\psi_1)=0\oplus (P_{H(\theta_1)}\psi_1)$ and
$\xi=(P_{H(\theta_0)}\phi_0)\oplus (P_{H(\theta_1)}\phi_1)$. It is
manifest in view of (ii) that $(\theta_1/\phi_1)|(\theta_0/\phi_0)$
and we  have trivially that $(\theta_0/\psi_0)|(\theta_1/\psi_1)$.
By Lemma \ref{cycliccontained}, we have that
$$
\bigvee_{n=0}^\infty T^n \xi\subset \left\{(P_{H(\theta_0)}\phi_0
g)\oplus (P_{H(\theta_1)}\phi_1 g): g\in H^2 \right\}.
$$
and
$$
\bigvee_{n=0}^\infty T^n \eta\subset \left\{(P_{H(\theta_0)}\psi_0
g)\oplus (P_{H(\theta_1)}\psi_1 g): g\in H^2 \right\}=\left\{0\oplus
(P_{H(\theta_1)}\psi_1 g): g\in H^2 \right\}.
$$
We want to show that the sets appearing on the right-hand sides
intersect trivially. Suppose then that $P_{H(\theta_0)}\phi_0 g=0$
and $P_{H(\theta_1)}\phi_1 g=P_{H(\theta_1)}\psi_1 h$ for some
$g,h\in H^2$. The first relation implies that $g\in
(\theta_0/\phi_0)H^2$. Since $(\theta_1/\phi_1)|(\theta_0/\phi_0)$,
we have that $\theta_1|(\phi_1 \theta_0/\phi_0)$ and thus
$$
P_{H(\theta_1)}\psi_1 h=P_{H(\theta_1)}\phi_1 g=0.
$$
We have therefore established that
$$
\left\{(P_{H(\theta_0)}\phi_0 g)\oplus (P_{H(\theta_1)}\phi_1 g): g\in H^2 \right\}\cap \left\{(P_{H(\theta_0)}\psi_0 g)\oplus (P_{H(\theta_1)}\psi_1 g): g\in H^2 \right\}=\{0\}
$$
whence $K_{\xi}\cap K_{\eta}=\{0\}$, where
$K_{\xi}=\bigvee_{n=0}^\infty T^n \xi$ and
$K_{\eta}=\bigvee_{n=0}^\infty T^n \eta$.

It is straightforward to verify that the minimal function of $\xi$
is $\alpha_0$ (by (ii)), and that the minimal function of $\eta$ is
$\alpha_1$. Define $N=K_{\xi}\vee K_{\eta}$. Using the fact that
$K_{\xi}\cap K_{\eta}=\{0\}$, it is easy to see that
$$
T|N\sim T|K_{\xi}\oplus T|K_{\eta}\sim S(\alpha_0)\oplus S(\alpha_1).
$$
It follows from (iii) that $\phi_0|\psi_0$ and $\phi_1|\psi_1$, thus
$\ol{P_{H(\theta_j)}N}=\phi_j H^2\ominus \theta_j H^2$ for $j=0,1$.
Finally, suppose the Jordan model of $T_{N^\perp}$ is equal to
$S(\gamma_0)\oplus S(\gamma_1)$. By Theorem \ref{jordanmodel} (along
with (ii) and (iii)), we find that $\gamma_1=\phi_1=\beta_1$. By
Theorem \ref{det} we have
$$
\theta_0\theta_1=\det T=\det (T|N) \det (T_{N^\perp})=\alpha_0 \alpha_1 \gamma_0 \beta_1
$$
so that property (iv) implies $\gamma_0=\beta_0$. This completes the proof.
\end{proof}

\section{Classification theorem}
We take the final step towards our classification result.
\begin{proposition}\label{main2}
Let $T=S(\theta_0)\oplus S(\theta_1)$ be a Jordan operator. Let $M$
and $M'$ be invariant subspaces for $T$ such that $T|M\sim T|M'$ and
$T_{M^\perp}\sim T_{M'^{\perp}}$. Let $S(\alpha_0)\oplus
S(\alpha_1)$ be the Jordan model of $T|M$ and $T|M'$, and
$S(\beta_0)\oplus S(\beta_1)$ be the Jordan model of $T_{M^\perp}$
and $T_{M'^{\perp}}$. For $j=0,1$, let
$$
\ol{P_{H(\theta_j)}M}=\phi_j H^2\ominus \theta_j H^2
$$
and
$$
\ol{P_{H(\theta_j)}M'}=\phi'_j H^2\ominus \theta_j H^2,
$$
where $\phi_j$ and $\phi_j'$ are inner divisors of $\theta_j$.
Assume that $\phi_1|\phi_0$ and
$(\theta_1/\phi_1)|(\theta_0/\phi_0)$, along with $\phi'_1|\phi'_0$
and $(\theta_1/\phi'_1)|(\theta_0/\phi'_0)$. Then, there exists a
quasiaffinity $X\in \{T\}'$ such that $M=\ol{XM'}$.
\end{proposition}
\begin{proof}
By Proposition \ref{jordanmodel}, we have that
$\phi_0\equiv\theta_0/\alpha_0\equiv\phi'_0$ and
$\phi_1\equiv\beta_1\equiv\phi'_1$, so that
$$
\ol{P_{H(\theta_j)}M}=\ol{P_{H(\theta_j)}M'}
$$ for $j=0,1$. Define
$$
E=\ol{P_{H(\theta_0)} M}\oplus \ol{P_{H(\theta_1)}M}=\ol{P_{H(\theta_0)} M'}\oplus
\ol{P_{H(\theta_1)} M'},
$$
which is hyperinvariant for $T$ by Theorem \ref{hyperinvJordan}.
Hence, $E$ contains $M$ and $M'$ along with any image of those
subspaces under an operator lying in the commutant of $T$.

By Theorem \ref{maxvector}, we may choose $\xi=\xi_0\oplus \xi_1 \in
M$ a maximal vector for $T|M$ with the additional property that
$\xi_j$ is maximal for $S(\theta_j)|\ol{P_{H(\theta_j)} M}$ for each
$j=0,1$. Similarly, we may choose  $\xi'=\xi'_0\oplus \xi'_1\in M'$
a maximal vector for $T|M'$ with the additional property that $
\xi'_j$ is maximal for $S(\theta_j)|\ol{P_{H(\theta_j)}M'}$ for each
$j=0,1$. By Lemma \ref{outer}, there exists an outer function $v\in
H^\infty$ such that $v\xi_0,v\xi_1,v\xi'_0$ and $v\xi'_1$ all belong
to $H^\infty$. Since $v$ is outer and $\theta_0$ is inner, we have
$v\wedge \theta_0\equiv 1$, so that $v(S(\theta_0)), v(S(\theta_1))$
and $v(T)$ are quasiaffinities by Proposition
\ref{jordanblockprops}, and thus $v(T)\xi$ and $v(T)\xi'$ have the
same maximality properties as $\xi$ and $\xi'$ respectively.
Consequently, upon replacing $\xi\in M$ and $\xi'\in M'$ by
$v(T)\xi\in M$ and $v(T)\xi'\in M'$ respectively, we may further
assume that
$$
\xi=(P_{H(\theta_0)}\phi_0 f_0) \oplus (P_{H(\theta_1)}\phi_1 f_1)
$$
and
$$
\xi'=(P_{H(\theta_0)}\phi_0 f'_0)\oplus (P_{H(\theta_1)}\phi_1 f'_1)
$$
where $f_0,f_1,f'_0,f'_1\in H^\infty$ satisfy $f_j\wedge
\theta_0\equiv f'_j\wedge \theta_0\equiv 1$ for $j=0,1$. Define
$$
Y=\left(
\begin{array}{cc}
f_0 (S(\theta_0)) & 0\\
0 & f_1(S(\theta_1))
\end{array}
\right),
Y'=\left(
\begin{array}{cc}
f'_0 (S(\theta_0)) & 0\\
0 & f'_1(S(\theta_1))
\end{array}
\right).
$$
It is clear $Y'\xi=Y\xi'$. Lemma \ref{invert} implies that $Y,Y'\in
\{T\}'$ are quasiaffinities and that there exist quasiaffinities
$Z,Z'\in \{T\}'$ such that $ZY=YZ=(f_0 f_1) (T)$ and
$Z'Y'=Y'Z'=(f'_0 f'_1) (T)$. If we set $K=\bigvee_{n=0}^\infty T^n
\xi \subset M$ and $K'=\bigvee_{n=0}^\infty T^n \xi'\subset M'$,
then $Z'Y\xi'=Z'Y'\xi=(f'_0 f'_1)(T)\xi \in K$, so that
$\ol{Z'YK'}\subset K$. By virtue of Theorem \ref{injection}, we see
that $T|\ol{Z'YK'}\sim T|K'\sim S(\alpha_0)$, so Theorem
\ref{invsubmult1} implies that $\ol{Z'YK'}=K$. In particular,
$K\subset \ol{Z'Y M'}$. In addition, we have $\ol{Z'Y M'}\subset E$
because of the hyperinvariance of $E$ .

Choose now  a cylic vector $k\in K$ for $(T|K)^*$; the fact that
$(T|K)^*$ has multiplicity one follows from Theorem
\ref{jordanadjoint}. Define the following subspaces
$$
K_{M}=\bigvee_{n=0}^\infty (T|M)^{*n}k
$$
$$
K_{M'}=\bigvee_{n=0}^\infty (T|\ol{Z'Y M'})^{*n}k
$$
$$
K_E=\bigvee_{n=0}^\infty (T|E)^{*n}k
$$
along with $L_M=M\ominus K_M, L_{M'}=\ol{Z'Y M'}\ominus K_{M'}$ and $ F=E\ominus K_E$. We claim now that $L_M\subset F$ and $L_{M'}\subset F$. Indeed, let $h\in L_M\subset M\subset E$. Then, for any $n\geq 0$, we have
\begin{align*}
\langle h, (T|E)^{*n}k\rangle &=\langle h, P_E T ^{*n}k\rangle\\
&=\langle h, P_M T ^{*n}k\rangle\\
&=\langle h, (T|M)^{*n}k\rangle\\
&=0
\end{align*}
since $L_M\perp K_M$. A similar computation shows that $L_{M'}\subset
F$.

Note that the Jordan model of $T|E$ is easy to identify
given the form of $E$ and the fact that $\theta_1/\phi_1$ divides
$\theta_0/\phi_0$:
$$
T|E\sim S(\theta_0/\phi_0)\oplus S(\theta_1/\phi_1)=S(\alpha_0)\oplus S(\theta_1/\beta_1).
$$
Also, by Theorem \ref{injection} we have $T|{\ol{Z'Y M'}}\sim
T|M'\sim S(\alpha_0)\oplus S(\alpha_1)$. We now apply Theorem
\ref{splitting} to $T|M, T|\ol{Z'Y M'}$ and $T|E$ for the subspace
$K\subset M\cap \ol{Z'Y M'} \cap E$. We have that
$$
L_M\cap K=L_{M'}\cap K=F\cap K=\{0\},
$$
$$
M=K\vee L_M, \ol{Z'Y M'}=K\vee L_{M'}, E=K\vee F $$ and
$$
T|L_M\sim T|L_{M'}\sim S(\alpha_1), T|F \sim S(\theta_1/\beta_1).
$$
Since
$L_M,L_{M'}\subset F$ and $T|F$ has multiplicity one, Theorem
\ref{invsubmult1} implies that $L_M=L_{M'}$, whence $M=K\vee L_M=K \vee
L_{M'}=\ol{Z'Y M'}$. Setting $X=Z'Y$ completes the proof.
\end{proof}

We can now establish the existence of a canonical space.
\begin{theorem}\label{similcanmodel}
Let $T\in \B(\hil)$ be an operator of class $C_0$ with Jordan model
$S(\theta_0)\oplus S(\theta_1)$.  Let $M\subset \hil$ be an
invariant subspace for $T$ such that $S(\alpha_0)\oplus S(\alpha_1)$
is the Jordan model of $T|M$ and $S(\beta_0)\oplus S(\beta_1)$ is
the Jordan model of $T_{M^\perp}$. Let $N\subset H(\theta_0)\oplus
H(\theta_1)$ be the smallest invariant subspace for
$S(\theta_0)\oplus S(\theta_1)$ containing
$$
\xi=(P_{H(\theta_0)}(\theta_0/\alpha_0))\oplus (P_{H(\theta_1)}\beta_1)
$$
and
$$
\eta=0\oplus (P_{H(\theta_1)} (\theta_1/\alpha_1)).
$$
Then, $M$ is quasisimilar to $N$.
\end{theorem}
\begin{proof}
Let $J=S(\theta_0)\oplus S(\theta_1)$. By Theorem \ref{existencejordan}, we can find quasiaffinities $A:\hil\to H(\theta_0)\oplus H(\theta_1)$ and $B:H(\theta_0)\oplus H(\theta_1)\to \hil$ such that $AT=JA, BJ=TB$.
Applying Theorem \ref{hyper}, we obtain a quasiaffinity $X\in \{J\}'$ such that
$$
\ol{P_{H(\theta_0)}XAM}\oplus \ol{P_{H(\theta_1)}XAM}
$$
is a hyperinvariant subspace for $J$. By virtue of Proposition \ref{jordansim}, we have
$J|\ol{XAM}\sim T|M$ and $J_{(XAM)^\perp}\sim T_{M^\perp}$. By Theorem \ref{hyperinvJordan}, Proposition \ref{jordanmodel}, Proposition \ref{canonicalspace} and Proposition \ref{main2}, we can find a quasiaffinity $Y\in\{J\}'$ such that $\ol{YXAM}=N$, and thus $M\prec N$. Now, it is obvious that $N\prec \ol{BN}$, whence $M\prec \ol{BN}$. But then Lemma \ref{similar} implies that $\ol{BN}\sim M$, and thus $N\prec M$. This completes the proof.
\end{proof}

The main result of the paper is now easily proved.

\begin{theorem}\label{classification}
Let $T$ be an operator of class $C_0$ with multiplicity
two.  Let $M,M'$ be two invariant subspaces for $T$.
Then, $M \sim M'$ if and only if $T|M\sim T|M'$ and $T_{M^\perp}\sim
T_{M'^\perp}$.
\end{theorem}
\begin{proof}
One direction follows from Proposition \ref{jordansim} and the
sentence immediately following its proof. Assume now that $T|M\sim
T|M'$ and $T_{M^\perp}\sim T_{M'^\perp}$. Then, Theorem
\ref{similcanmodel} implies the existence of an invariant subspace
$N$ for the Jordan model of $T$ such that $M\sim N$ and $M'\sim N$,
so we are done.
\end{proof}

\section{Acknowledgements}
The author was supported by a NSERC PGS grant.

\end{document}